\documentclass[12pt]{article}

\usepackage{amsthm}
\usepackage{amsmath}
\usepackage{bbm}
 \usepackage{amssymb} 

\newcommand{\Exp}{{\rm I\hspace{-0.8mm}E}}
\newcommand{\Prob}{{\rm I\hspace{-0.8mm}P}}

\newcommand{\iz}{{\rm \rlap Z\kern 2.2pt Z}}

\newcommand{\ind}{{1\hspace{-1mm}{\rm I}}}

\newtheorem{theorem}{Theorem}

\newtheorem{proposition}{Proposition}

\newtheorem{remark}{Remark}
\newtheorem{example}{Example}

\title{{\bf  Formula for the supremum distribution 
of a spectrally positive L\'evy process}}
\author{\\
Zbigniew Michna\\
Department of Mathematics and Cybernetics\\
Wroc{\l}aw University of Economics\\
Wroc{\l}aw Poland}
\date{}

\begin{document}

\baselineskip=17pt

\maketitle
%\vspace{2cm}

\begin{abstract}
In this article we derive formula for probability $\Prob(\sup_{t\leq T} (X(t)-ct)>u)$
where $X=\{X(t)\}$ is a spectrally positive L\'evy process and
$c\in\mathbb{R}$.
As an example we investigate the inverse Gaussian L\'evy process.

\vspace{5mm}
\newline
\vspace{2cm}
\end{abstract}

\section{Introduction} 
L\'evy processes arise in many areas of probability and play an important role among stochastic processes. Moreover they
are a basic skeleton in many applications in finance, insurance, queueing systems and engineering.
Especially extreme values of stochastic processes are applicable in many fields. 
In this article we will consider spectrally positive L\'evy processes
that is L\'evy processes with L\'evy measure concentrated on $(0,\infty)$. Let $X=\{X(t): t\geq 0\}$ be a L\'evy 
process with characteristic function of the form
\begin{equation}\label{chl}
\Exp\exp\{iuX(t)\}=\exp\left\{t\left[iau+\int_0^\infty(e^{iux}-1-iux\ind(x< 1))\,Q(dx) \right]\right\}
\end{equation}
where $a\in \mathbb{R}$ and $Q$ is a L\'evy measure (this is a pure jump L\'evy process without Gaussian component). We assume
that the L\'evy measure $Q$ has a density bounded on $[\epsilon,\infty)$ for every $\epsilon>0$.
We will investigate two cases that is the finite variation case and the compensated case. In the first case
we assume $\int_0^1 x\,Q(dx)<\infty$ and the characteristic function of $X$ we put 
\begin{equation}\label{chlf}
\Exp\exp\{iuX(t)\}=\exp\left\{t\left[\int_0^\infty(e^{iux}-1)\,Q(dx) \right]\right\}\,.
\end{equation}
In the compensation case we have $\int_0^1 x\,Q(dx)=\infty$.

Exact formulas for the supremum distribution of a stochastic process 
are known only in a few cases: Brownian motion, compound Poisson process, subordinators (see Tak\'acs \cite{ta:65}), gamma process (see Dickson and Waters \cite{di:wa:93}), stable L\'evy process (see Bernyk et al. \cite{be:da:pe:08} and Michna \cite{mi:11}). In Tak\'acs \cite{ta:65}
a generalization of the classical ballot theorem is provided to get the distribution of the supremum
for the processes with interchangeable non-negative increments.

In this article we will find the exact formula for the following probability 
\begin{equation}\label{psi}
\Psi(u,T)=\Prob(\sup_{t\leq T} (X(t)-ct)>u)
\end{equation}
where $T>0$, $u\geq 0$ and $c\in\mathbb{R}$. For $c>0$ probability (\ref{psi}) is the so-called finite time ruin probability. In the compensation case the drift $c$ can be included in the parameter $a$.
In the articles Dickson and Waters \cite{di:wa:93} and Michna \cite{mi:11} a special approximation of gamma process and $\alpha$-stable L\'evy process by compound Poisson processes is used. In Michna \cite{mi:11} central limit theorem is
applied to approximate $\alpha$-stable L\'evy processes. The method to derive the formula for the supremum used in this paper is similar to the method applied in Dickson and Waters \cite{di:wa:93} and Michna \cite{mi:11} but here we apply a general 
approximation of a L\'evy process by compound Poisson processes which requires different arguments in the proof and get formulas for the probability (\ref{psi}) which in the compensation case seem to be a new result.

We will use the results for the so-called finite time ruin probability for the classical risk process 
(the compound Poisson process). First let
$$
S(t)=\sum_{k=1}^{N(t)}Y_k -ct
$$ 
be the classical claim surplus process, where $\{Y_k\}_{k=1}^\infty$ is an iid sequence of positive random variables, $N(t)$ is a Poisson process and $c>0$.
The first formula below is going back to Cram\'er (see e.g. Asmussen and Albrecher\cite{as:00}) and the second one is the so-called Seal's formula (see Seal \cite{se:74}, originating from Prabhu \cite{pr:61}). Let
$\psi(u,T)=\Prob(\sup_{t\leq T} S(t)> u)$.
\begin{theorem}\label{psirisk} 
For $T>0$, $u>0$ and $c>0$ we have
\begin{enumerate}
\item[(i)]
$$
1-\psi(0,T)=\Prob(\sup_{t\leq T} S(t)= 0)
=\frac{1}{cT}\int_0^{cT}\Prob\left(\sum_{k=1}^{N(T)}Y_k\leq x\right)\,dx\,,
$$
\item[(ii)]
\begin{eqnarray*}
\lefteqn{1-\psi(u,T)=\Prob(\sup_{t\leq T} S(t)\leq u)}\\
&=&\Prob(S(T)\leq u)-c\int_0^{T}(1-\psi(0,T-s))
\Prob(S(s)\in [u,u+ds])\\
&=&\Prob(S(T)\leq u)-c\int_0^{T}(1-\psi(0,T-s))f(u+cs,s)\,ds\,,
\end{eqnarray*}
where $f(x,s)$ is the density function of $\sum_{k=1}^{N(s)}Y_k$, provided that it exists.
\end{enumerate}
\end{theorem}
\begin{remark}
A compound Poisson distribution has an atom at $0$ thus the above function $f$ is a density of the positive part of the distribution support. 
\end{remark}

We will apply approximation of a L\'evy process $X$ by compound Poisson processes.
Let us put 
$$
N_\epsilon(t)=\sum_{s\leq t}\Delta X(s)\ind(\Delta X(s)\geq \epsilon)\,.
$$
The process $N_\epsilon$ is a compound Poisson process with positive jumps.
\begin{proposition}\label{weak}
Let $X$ be a spectrally positive L\'evy process with a L\'evy measure $Q$. Then
\begin{enumerate}
\item[(i)]
for $\int_0^1 x\,Q(dx)<\infty$ 
\begin{equation}\label{sn1}
N_\epsilon(t)\rightarrow X(t)\,;
\end{equation}
\item[(ii)]
for $\int_0^1 x\,Q(dx)=\infty$
\begin{equation}\label{sn2}
N_\epsilon(t)+(a-\int_{\epsilon}^1x\,Q(dx))t\rightarrow X(t)\,,
\end{equation}
\end{enumerate}
as $\epsilon\downarrow 0$ a.s. in the uniform topology.
\end{proposition}
\begin{proof}
The assertions follow from L\'evy-It\^o representation see e.g. Sato \cite{sa:99}.
\end{proof}

\section{Main result}
Now we use the above results to derive a formula for probability (\ref{psi}) 
(for a similar treatment for gamma process see Dickson and Waters \cite{di:wa:93} and for $\alpha$-stable case Michna \cite{mi:11}).
\begin{theorem}\label{main}
Let $X$ be a spectrally positive L\'evy process, $T>0$ and $u>0$ then
\begin{enumerate}
\item[(i)]
for $\int_0^1 x\,Q(dx)<\infty$ and $c>0$
$$
\Psi(0,T)=\Prob(\sup_{t\leq T} (X(t)-ct)>0)=1-\frac{1}{cT}\int_0^{cT}\Prob(X(T)\leq x)\,dx
$$
and
\begin{eqnarray*}
\Psi(u,T)&=&\Prob(\sup_{t\leq T} (X(t)-ct)>u)\\
&=&\Prob(X(T)-cT > u)+\\
&&\int_0^T\frac{f(u+cs,s)}{T-s}
\,ds\int_{0}^{c(T-s)}\Prob(X(T-s)\leq x)\,dx\,,
\end{eqnarray*}
where $f(x,s)$ is the density function of $X(s)$, provided that it exists;
\item[(ii)]
for $\int_0^1 x\,Q(dx)=\infty$ and $a\in\mathbb{R}$ 
$$
\Psi(0,T)=\Prob(\sup_{t\leq T} X(t)>0)=1
$$
and
\begin{eqnarray}
\Psi(u,T)&=&\Prob(\sup_{t\leq T} X(t)>u)\nonumber\\
&=&\Prob(X(T) > u)+\nonumber\\
&&\int_0^T\frac{f(u,s)}{T-s}
\,ds\int_{-\infty}^{0}\Prob(X(T-s)\leq x)\,dx\,,\label{gensup}
\end{eqnarray}
where $f(x,s)$ is the density function of $X(s)$, provided that it exists and the cumulative distribution of $\sup_{t\leq T}X(t)$ is continuous.
\end{enumerate}
\end{theorem}
\begin{remark}
The formula in (i) is similar to the formula of Tak\'acs \cite{ta:65} valid for the processes with interchangeable non-negative increments.
\end{remark}
\begin{remark}
The formula for $\Psi(u,T)$ can be unified for the both cases because in the first case the support of the distribution
of $X(t)$ is in $[0,\infty)$ and the formula (\ref{gensup}) is valid for the case (i) too.
\end{remark}
\begin{remark}
Formula (\ref{gensup}) is valid for Wiener process too (see Michna  \cite{mi:11}).
\end{remark}
\begin{remark}
In the case (i) the distribution of $\sup_{t\leq T}(X(t)-ct)$ has an atom at $u=0$.
\end{remark}
\begin{remark}
The case (i) for $c\leq 0$ is trivial.
\end{remark}
\begin{proof}
First we show the assertions of (i).
Let $\{\epsilon_n\}_{n=1}^\infty$ be a decreasing sequence of positive numbers such that $\lim_{n\rightarrow\infty}\epsilon_n=0$.
Then it is easy to notice that \\
$N_{\epsilon_n}(t)\leq N_{\epsilon_{n+1}}(t)$
for all $t\geq 0$ and $n$. Thus 
$$
\sup_{t\leq T}(N_{\epsilon_n}(t)-ct)\leq \sup_{t\leq T}(N_{\epsilon_{n+1}}(t)-ct)
$$
and
$$
\lim_{n\rightarrow\infty}\sup_{t\leq T}(N_{\epsilon_n}(t)-ct)=
\sup_{t\leq T}(X(t)-ct)\,.
$$
Hence for all $u\in\mathbb{R}$ we have
$$
\lim_{n\rightarrow\infty}\Prob(\sup_{t\leq T}(N_{\epsilon_n}(t)-ct)\leq u)=
\Prob(\sup_{t\leq T}(X(t)-ct)\leq u)\,.
$$
Using (i) of Theorem \ref{psirisk} we get for $c>0$ that
\begin{eqnarray}
1-\psi_\epsilon(0,T)&=&\Prob(\sup_{t\leq T} (N_\epsilon(t)-ct)\leq 0)\nonumber\\
&=&\frac{1}{cT}\int_0^{cT}
\Prob\left(N_\epsilon(T)\leq x\right)\,dx\label{psin}\\
&\rightarrow &\frac{1}{cT}\int_0^{cT}\Prob\left(X(T)\leq x\right)\,dx  \nonumber
\end{eqnarray}
where in the last line $\epsilon\downarrow 0$.
Thus we get the first assertion of (i).

Now let us consider (ii) of Theorem \ref{psirisk} for the process $N_\epsilon$. The random variable $N_\epsilon(t)$ has a density by the assumption that 
the L\'evy measure $Q$ has a bounded density on $[\epsilon,\infty)$. Thus we have for $c>0$ and $u>0$
\begin{eqnarray*}
\lefteqn{1-\psi_\epsilon(u,T)=\Prob(\sup_{t\leq T} (N_\epsilon(t)-ct)\leq u)}\\
&=&\Prob(N_\epsilon(T)-cT\leq u)-c\int_0^{T}(1-\psi_\epsilon(0,T-s))\,
\Prob(N_\epsilon(s)-cs\in [u,u+ds])\\
&=&\Prob(N_\epsilon(T)-cT\leq u)-\\
&&\,\,\,\,\,\,\int_0^{T}\frac{1}{T-s}\,\Prob(N_\epsilon(s)-cs\in [u,u+ds])\int_0^{c(T-s)}\Prob\left(N_\epsilon(T-s)\leq x\right)\,dx\,\\
&\rightarrow&\Prob(X(T)-cT\leq u)-\\
&&\,\,\,\,\,\,\int_0^{T}\frac{1}{T-s}\,\Prob(X(s)-cs\in [u,u+ds])\,\int_0^{c(T-s)}\Prob\left(X(T-s)\leq x\right)\,dx\,,
\end{eqnarray*}
where in the second last equality we substituted (\ref{psin}) and in the last one $\epsilon\downarrow 0$.
This way we obtain the second assertion of (i).

Now let us prove the assertions of (ii). We put $c_\epsilon=-a+\int_{\epsilon}^1x\,Q(dx)$ and
notice that $\lim_{\epsilon\rightarrow 0}c_\epsilon=\infty$. 
By the property (ii) of Proposition \ref{weak} we have
$$
\lim_{\epsilon\downarrow 0}\sup_{t\leq T}(N_{\epsilon}(t)-c_\epsilon t)=
\sup_{t\leq T}X(t)
$$
which implies
that 
$$
\lim_{\epsilon\downarrow 0}\Prob(\sup_{t\leq T}(N_{\epsilon}(t)-c_\epsilon t)\leq u)=
\Prob(\sup_{t\leq T}X(t)\leq u)\,
$$
for every $u\in\mathbb{R}$ where $\Prob(\sup_{t\leq T}X(t)\leq u)$ is a continuous function of $u$.
Thus for a sufficiently small $\epsilon$ we have
$c_\epsilon>0$ and we get
\begin{eqnarray}
1-\psi_\epsilon(0,T)&=&\Prob(\sup_{t\leq T} (N_\epsilon(t)-c_\epsilon t)\leq 0)\nonumber\\
&=&\frac{1}{c_\epsilon T}\int_0^{c_\epsilon T}
\Prob\left(N_\epsilon (T)\leq x\right)\,dx\nonumber\\
&=&\frac{1}{c_\epsilon T}\int_{-c_\epsilon T}^0
\Prob\left(N_\epsilon (T)-c_\epsilon T\leq x\right)\,dx\label{psin2}\\
&\rightarrow& 0 \nonumber
\end{eqnarray}
where in the last step $\epsilon\downarrow 0$ and the last integral tends to \\
$\int_{-\infty}^0\Prob\left(X(T)\leq x\right)\,dx$ which is finite by the fact the
L\'evy measure $Q$ is concentrated on $(0,\infty)$. For $u>0$ we obtain
\begin{eqnarray*}
\lefteqn{1-\psi_\epsilon(u,T)=\Prob(\sup_{t\leq T} (N_\epsilon(t)-c_\epsilon t)\leq u)}\\
&=&\Prob(N_\epsilon(T)-c_\epsilon T\leq u)-c_\epsilon\int_0^{T}(1-\psi_\epsilon(0,T-s))
\Prob(N_\epsilon(s)-c_\epsilon s\in [u,u+ds])\\
&=&\Prob(N_\epsilon(T)-c_\epsilon T\leq u)-\\
&&\,\,\,\,\,\int_0^{T}
\frac{1}{T-s}\,\Prob(N_\epsilon(s)-c_\epsilon s\in [u,u+ds])\cdot\\
&&\,\,\,\,\,\,\,\,\,\,\,\,\,\int_{-c_\epsilon(T-s)}^0
\Prob\left(N_\epsilon(T-s)-c_\epsilon (T-s)\leq x\right)\,dx\,\\
&\rightarrow&\Prob(X(T)\leq u)-\\
&&\,\,\,\,\,\int_0^{T}
\frac{f(u,s)}{T-s}\,ds\int_{-\infty}^0
\Prob\left(X(T-s)\leq x\right)\,dx\,,
\end{eqnarray*}
where in the second last equality we substituted (\ref{psin2}) and in the last one $\epsilon\downarrow 0$.
\end{proof}

We can give the following sufficient condition for the L\'evy process $X$ to have the cumulative distribution of $\sup_{t\leq T}(X(t)-ct)$ continuous.
\begin{proposition}
If $X$ is a L\'evy process such that 
\begin{equation}\label{limsup}
\limsup_{t\downarrow 0} \frac{X(t)}{t^{\kappa}}>0
\end{equation}
a.s. for $\kappa<1$ then the cumulative distribution of $\sup_{t\leq T}(X(t)-ct)$ is continuous.
\end{proposition}
\begin{remark}
Conditions for the property (\ref{limsup}) are given in Bertoin et al. \cite{be:do:ma:08}.
\end{remark}
\begin{proof}
By (\ref{limsup}) we get
$$
\limsup_{t\downarrow 0} \frac{X(t)-ct}{t^{\kappa}}>0
$$
a.s. which gives that the process $X(t)-ct$ is above the barrier $u=0$ infinitely
many times in every neighborhood of $t=0$ a.s. Thus 
\begin{equation}\label{supz}
\Prob(\sup_{t\leq T}(X(t)-ct)=0)=0\,.
\end{equation}
For $u>0$ we define the stopping time $\tau(u)=\inf\{t>0: X(t)>u\}$ and by the strong Markov property
we get that the process $Y(t)=X(\tau+t)-X(\tau)$ is a L\'evy process with the same distribution as
the process $X$ on the set $\{\tau<\infty\}$. Thus using (\ref{supz}) we obtain that
$$
\Prob(\sup_{t\leq T}(X(t)-ct)=u)=0\,.
$$
\end{proof}

For stable L\'evy processes the continuity of the cumulative distribution of \\
$\sup_{t\leq T}(X(t)-ct)$ can be shown using the law of the iterated logarithm for small times (see Zolotarev \cite{zo:64}).
Moreover Theorem \ref{main} can be applied to many other L\'evy processes such as the spectrally positive generalized hyperbolic processes (e.g. the inverse Gaussian L\'evy process) 
introduced by Barndorff-Nielsen and Halgareen \cite{ba:ha:77}, 
the spectrally positive L\'evy processes proposed by Carr et al. \cite {ca:ge:ma:yo:03} or tempered stable processes introduced by Koponen \cite{ko:95}
and generalized by Rosi\'nski \cite{ro:07}. Recently the inverse Gaussian
L\'evy process is extensively used in financial modeling. Thus let us derive the exact formula for the supremum distribution
of the inverse Gaussian process.

\begin{example}
Let us consider the inverse Gaussian process $X(t)$ with the following
density 
$$
f(x,t)=\frac{\delta t e^{\gamma\delta t}}{\sqrt{2\pi}}x^{-3/2}\exp\left\{-\frac{1}{2}(\delta^2 t^2 x^{-1}+\gamma^2 x)\right\}\ind\{x>0\}
$$
where $\delta>0$ and $\gamma>0$. The process $X$ is a subordinator. By Theorem \ref{main} (i) we have for $c>0$
\begin{eqnarray*}
\lefteqn{\Prob(\sup_{t\leq T}(X(t)-ct)> 0)}\\
&=&1-\frac{\delta e^{\gamma\delta T}}{c\sqrt{2\pi}}\int_0^{cT}dx\int_0^x y^{-3/2}\exp\left\{-\frac{1}{2}(\delta^2T^2y^{-1}+\gamma^2 y)\right\}\,dy\\
&=& 1-\frac{\delta e^{\gamma\delta T}}{c\sqrt{2\pi}}\int_0^{cT}(cT x^{-3/2}-x^{-1/2})\exp\left\{-\frac{1}{2}(\delta^2T^2x^{-1}+\gamma^2 x)\right\}\,dx\,,
\end{eqnarray*}
where in the last step we integrate by parts. For $u>0$ we consider the second assertion of Theorem \ref{main} (i) and similarly we obtain
\begin{eqnarray*}
\lefteqn{\Prob(\sup_{t\leq T}(X(t)-ct)> u)}\\
&=&\frac{\delta Te^{\gamma\delta T}}{\sqrt{2\pi}}\int_{u+cT}^{\infty}x^{-3/2}\exp\left\{-\frac{1}{2}(\delta^2 T^2x^{-1}+\gamma^2 x)\right\}\,dx\\
&&+\frac{\delta^2 e^{\gamma\delta T}}{2\pi}\int_0^{T}s(u+cs)^{-3/2}\exp\left\{-\frac{1}{2}(\delta^2s^2(u+cs)^{-1}+\gamma^2 (u+cs))\right\}\,ds\\
&&\,\,\,\,\,\,\cdot\int_0^{c(T-s)}(c(T-s) x^{-3/2}-x^{-1/2})\exp\left\{-\frac{1}{2}(\delta^2(T-s)^2x^{-1}+\gamma^2 x)\right\}\,dx\,.
\end{eqnarray*}
\end{example}

One can investigate a tempered stable L\'evy process $X$ with the L\'evy measure 
$$
Q(dx)=\frac{A e^{-\lambda x}}{x^{\alpha+1}}\ind\{x>0\}\,dx\,,
$$ 
where $A>0$, $\lambda>0$ and $0<\alpha<2$. The closed form of the density function of $X(t)$ is not known but the characteristic function
of the distribution of $X(t)$ can be given in a closed form see Carr et al. \cite{ca:ge:ma:yo:02}. Thus we can calculate the inverse
Fourier transform numerically to get the density function of $X(t)$. For $0<\alpha<1$ we apply Theorem \ref{main} (i) and for $1\leq\alpha<2$ we use
Theorem \ref{main} (ii) to get the formula for the supremum distribution of the process $X$.

\bibliographystyle{plainnat}

\end{document}